\documentclass[10pt]{amsart}
\usepackage{amsmath,amscd,amssymb,epsfig}

\newcommand{\bbz}{\mathbb{Z}}

\newcommand{\bbr}{\mathbb{R}}

\newcommand{\tA}{{\widetilde{A}}}
\newcommand{\tC}{{\widetilde{C}}}
\newcommand{\diam}{{\rm{Diam}}}

\newcommand{\dist}{\operatorname{dist}}

\newtheorem{theorem}{Theorem}[section]
\newtheorem{corollary}[theorem]{Corollary}
\newtheorem{proposition}[theorem]{Proposition}

\newtheorem{lemma}[theorem]{Lemma}
\newtheorem{defn}[theorem]{Definition}

\newtheorem{remark}[theorem]{Remark}

\newtheorem{observation}[theorem]{Observation}
\newtheorem{claim}[theorem]{Claim}

\numberwithin{figure}{section}

\begin{document}
\title[Triangle-Free Triangulations]{Triangle-Free Triangulations}

\author{Ron M.\ Adin}
\address{Department of Mathematics\\ 
Bar-Ilan University\\
Ramat-Gan 52900\\
Israel}
\email{radin@math.biu.ac.il} 

\author{Marcelo Firer} 
\address{Department of Mathematics\\ 
Institute of Mathematics, Statistics and Computer Science (IMECC)\\
State University of Campinas (UNICAMP)\\
13.081 Campinas SP\\
Brazil}
\email{mfirer@ime.unicamp.br}

\author{Yuval Roichman}
\address{Department of Mathematics\\
Bar-Ilan University\\
Ramat-Gan 52900\\
Israel}
\email{yuvalr@math.biu.ac.il} 


\date{January 27, 2009}

\begin{abstract}
The flip operation on colored inner-triangle-free triangulations
of a convex polygon is studied. It is shown that the affine Weyl
group $\widetilde{C}_n$ acts transitively on these triangulations
by colored flips, and that the resulting colored flip graph is 
closely related to a lower interval in the weak order on $\widetilde{C}_n$. 
Lattice properties of this order are then applied to compute the diameter.

\end{abstract}

\maketitle

\tableofcontents

\newpage

\section{Introduction}


In a seminal paper, using volume computations in hyperbolic geometry,
Sleator, Tarjan and Thurston~\cite{STT} computed the diameter of the 
flip graph of all triangulations of a convex polygon.
For special classes of triangulations, the diameter problem -- 
and, sometimes, even the question of connectivity -- is still open.
For instance, monochromatic-triangle-free triangulations were introduced by Propp. 
Sagan~\cite{Sa1} showed that the corresponding flip graph is connected only if
two colors are used. The diameter in this case is not known.

This paper studies the set of inner-triangle-free triangulations,
which is contained in the set of monochromatic-triangle-free triangulations.  
Methods from Coxeter group theory are applied to describe the structure of 
the resulting colored flip graph and to compute its diameter.
It is shown that the affine Weyl group $\widetilde{C}_n$
acts transitively, by flips, on such triangulations.
The stabilizer is computed, leading to an interpretation of the flip graph
as a Schreier graph. This graph is closely related to 
a distinguished lower interval in the weak order on $\widetilde{C}_n$.
Lattice properties of this order are then applied to compute the diameter.

\section{Basic Concepts}



Label the vertices of a convex $(n+4)$-gon $P_{n+4}$ ($n > 0$) by
the elements $0,\ldots,n+3$ of the additive cyclic group $\bbz_{n+4}$.
Consider a triangulation (with no extra vertices) of the polygon.
Each edge of the polygon is called an {\em external edge} of the triangulation; 
all other edges of the triangulation are called {\em internal edges}, or {\em chords}.

\begin{defn}
A triangulation of a convex $(n+4)$-gon $P_{n+4}$ is called 
{\em inner-triangle-free} (or simply {\em triangle-free})
if it contains no triangle with $3$ internal edges.
The set of all triangle-free triangulations of $P_{n+4}$ is denoted $TFT(n)$.
\end{defn}

\begin{defn}
A chord in $P_{n+4}$ is called {\em short} if it connects the
vertices labeled $i-1$ and $i+1$, for some $i\in \bbz_{n+4}$.
\end{defn}

\begin{claim}\label{t.short-chord}
For $n>0$, a triangulation of $P_{n+4}$ is triangle-free if and only if
it contains only two short chords.
\end{claim}

\begin{proof}
Any triangulation of $P_{n+4}$ consists of $n+1$ diagonals and $n+2$ triangles.
Each chord lies in exactly $2$ triangles. Thus the average number
of chords per triangle is $\frac{2(n+1)}{n+2} = 2-\frac{2}{n+2}$.
By definition, a triangulation is triangle-free if and only if each triangle
contains at most $2$ chords. On the other hand, each triangle contains at
least one chord. One concludes that there are exactly two triangles
each containing only one chord, completing the proof.
\end{proof}

\begin{defn}
A {\em proper coloring} of a triangulation $T\in TFT(n)$ is a
labeling of the chords by $0,\dots,n$ in the following
inductive way: Choose a short chord and label it $0$. Inductively,
a chord which was not yet labeled and is contained in a triangle whose
other chord has been labeled $i$, is labeled $i+1$.

It is easy to see that this uniquely defines the coloring.
The set of all properly colored triangle-free triangulations is denoted
$CTFT(n)$.
\end{defn}

\begin{defn}
Each chord in a triangulation is a diagonal of a unique quadrangle
(the union of two adjacent triangles). Replacing this chord by the
other diagonal of that quadrangle is a {\em flip} of the chord.

The {\em colored flip graph} $\Gamma_n$ is defined as follows:
the nodes are all the colored triangle-free triangulations in
$CTFT(n)$. Two triangulations are connected in $\Gamma_n$ by an
arc colored $i$ if one is obtained from the other by a flip of the
chord labeled $i$.
\end{defn}

\medskip


By Claim~\ref{t.short-chord},

\begin{corollary}\label{t.CTFTeq2TFT}
For $n>0$, any triangle-free triangulation of $P_{n+4}$ has exactly two
proper colorings. In other words,
$$
\#CTFT(n) = 2\cdot\#TFT(n).
$$
\end{corollary}

\begin{defn}
Define a map
$$
\varphi:CTFT(n) \to \bbz_{n+4}\times \bbz_2^{n}
$$
as follows: Let $T\in CTFT(n)$. If the (short) chord labeled $0$
in $T$ is $[a-1,a+1]$ for $a\in \bbz_{n+4}$, let $\varphi(T)_0:= a$. 
For $1\le i\le n$, assume that the chord labeled $i-1$ in $T$
is $[a-k,a+m]$ for some $k,m\ge 1$, $k+m = i+1$. 
The chord labeled $i$ is then either $[a-k-1,a+m]$ or $[a-k,a+m+1]$. 
Let $\varphi(T)_{i}$ be $0$ in the former case and $1$ in the latter.
\end{defn}

\begin{observation}\label{bijection-varphi}
$\varphi$ is a bijection.
\end{observation}


\begin{corollary}\label{t.nTFT}
For $n>0$, the number of triangle-free triangulations of a convex $(n+4)$-gon is
$$
\#CTFT(n)=(n+4)\cdot 2^{n}.
$$
\end{corollary}

\section{Group Action by Flips}

In this section we assume that $n>1$.

\subsection{The $\tC_n$-Action}

Let $\tC_n$ be the affine Weyl group generated by
$$
S=\{s_0,s_1,\ldots, s_{n-1},s_n\}
$$ 
subject to the Coxeter relations
\begin{equation}\label{relation1}
s_i^2=1\qquad (\forall i),
\end{equation}
\begin{equation}\label{relation2}
(s_is_j)^2=1\qquad (|j-i|>1),
\end{equation}
\begin{equation}\label{relation3}
(s_is_{i+1})^3=1\qquad  (1\le i<n-1),
\end{equation}
and
\begin{equation}\label{relation4}
(s_is_{i+1})^4=1 \qquad (i=0,n-1).
\end{equation}

The 
group $\tC_n$ acts naturally on $CTFT(n)$ by
flips: generator $s_i$ flips the chord labeled $i$ in $T\in
CTFT(n)$, provided that the resulting colored triangulation still
belongs to $CTFT(n)$. If this is not the case, $T$ is unchanged by
$s_i$.

Notice that $s_i(T)=T$ if and only if
$\varphi(T)_{i}=\varphi(T)_{i+1}$; also the only short chords are
labeled by $0$ and $n$, hence $s_0$ and $s_n$ never leave the
corresponding chords unchanged. Furthermore, one can easily verify
that by definition, the following observation holds.

\begin{observation}\label{action-on-vectors}
For every  $T\in CTFT(n)$ 
$$
(\varphi(s_0T))_j=\begin{cases} \varphi(T)_j,&\ {\rm{if}}\ j\ne 0,1,\\
\varphi(T)_0+1 \mod n+4, &\ {\rm{if}}\  j= 0 \ {\rm{and}}\ \varphi(T)_1=0, \\
\varphi(T)_0-1 \mod n+4, &\ {\rm{if}}\  j= 0 \ {\rm{and}}\ \varphi(T)_1=1, \\
\varphi(T)_1+1 \mod 2, &\ {\rm{if}}\  j= 1 \ {\rm{and}}\ \varphi(T)_1=0; \\
\end{cases}
$$
$$
(\varphi(s_nT)_j=\begin{cases} \varphi(T)_j,&\ {\rm{if}}\  j\ne n,00\\
\varphi(T)_n+1 \mod 2, &\ {\rm{if}}\  j= n;
\end{cases}
$$
and
$$
(\varphi(s_i T)_j=\varphi(T)_{s_i(j)}\qquad (0<i<n),
$$
\end{observation}

\begin{proposition}\label{t.action}
This operation determines a transitive $\tC_n$-action $CTFT(n)$.
\end{proposition}

\begin{proof}
To prove that the operation is a $\tC_n$-action, 
it suffices to show that it is consistent with the defining
Coxeter relations of $\widetilde{{C}}_{n}$.
Indeed, for every $i$, $s_{i}$ acts on a particular triangulation
$T\in CTFT(n)$ by flipping the diagonal labeled by $i$ or leaving
it unchanged; in both cases $s_i^2(T)=T$.
If $\left\vert j-i\right\vert
>1$, $s_{i}$ and $s_{j}$ act on diagonals of quadrangles with no
common triangle, hence $s_{i}$
and $s_{j}$ commute. 
Thus relation (\ref{relation2}) is satisfied.
Relations (\ref{relation3}) and (\ref{relation4}) may be verified
by a direct calculation of the corresponding flip operation,
taking in account the relative position of the relevant chords.
Alternatively, all relations may be easily verified using
Observation~\ref{action-on-vectors}. We leave verification of the
details to the reader.

\smallskip

To prove that the action is transitive, notice first that $s_0$
changes the location of the chord labeled by $0$, where the
cyclic orientation of this change depends on the relative position
of the chord labeled by $1$. It, thus, suffices to prove that the
maximal parabolic subgroup of $\tC_n$, $\langle
s_1,\dots,s_n\rangle$ acts transitively on all colored
triangle-free triangulations with a given $0$ chord. Indeed, the
parabolic subgroup  $\langle s_1,\dots,s_n\rangle$ is isomorphic
to the classical Weyl group $B_n$. By
Observation~\ref{action-on-vectors}, the restricted $B_n$-action
on all colored triangle-free triangulations with a given $0$
chord, may be identified with the natural $B_n$-action on all
subsets of $\{1,\dots,n\}$, and is thus transitive.

\end{proof}

\subsection{Stabilizer}\label{subsection-stab}

Define
$$
g_0 := s_0 s_1 \cdots s_{n-2} s_n s_{n-1} s_n s_{n-2} \cdots s_1 s_0 \in \tC_n
$$
and
$$
g_n := (s_n \cdots s_0)^{n+4} \in \tC_n.
$$
Denote:
$$
T_0:=\varphi^{-1}(0,\dots,0),
$$
the {\em canonical colored star triangulation}.

\begin{theorem}\label{t.stabilizer}
The subgroup $St_n = \langle g_0,s_1,\ldots,s_{n-1},g_n \rangle$ of $\tC_n$
is the stabilizer, under the $\tC_n$-action on $CTFT(n)$, of the canonical
colored star triangulation $T_0$. 
Stabilizers of other colored triangulations are subgroups of $\tC_n$ 
conjugate to $St_n$.
\end{theorem}

\begin{proof}
We shall proceed by a volume argument, using a sequence of technical observations. 

Consider the action of $\tC_n$ on $\bbr^n$ given by
$$
s_0(x_1,x_2,\ldots,x_n) := (-x_1,x_2,\ldots,x_n),
$$
$$
s_i(\ldots,x_i,x_{i+1},\ldots) := (\ldots, x_{i+1}, x_i, \ldots)\qquad(1\le i\le n-1),
$$
$$
s_n(x_1,\ldots,x_{n-1},x_n) := (x_1,\ldots,x_{n-1},2-x_n).
$$
It is well-known that this gives rise to a faithful $n$-dimensional linear
representation of $\tC_n$ (the natural action of $\tC_n$ on its root space).
The reflecting hyperplanes for the reflections $s_i$ are
$$
H_0 := \{(x_1,\ldots,x_n)\in\bbr^n\,|\,x_1 = 0\},
$$
$$
H_i := \{(x_1,\ldots,x_n)\in\bbr^n\,|\,x_i = x_{i+1}\}\qquad(1\le i\le n-1),
$$
$$
H_n := \{(x_1,\ldots,x_n)\in\bbr^n\,|\,x_n = 1\}.
$$
\begin{observation}\label{t.Fund_1}
A fundamental region for the above action of $\tC_n$ is the $n$-dimensional
simplex $Fund_1$ with vertices $v_0,\ldots,v_n\in \bbr^n$, where
$$
v_i := (\underbrace{0,\ldots,0}_{i},\underbrace{1,\ldots,1}_{n-i})\qquad(0\le i\le n).
$$
\end{observation}

$v_i$ is the intersection point of the reflecting hyperplanes for all generators
of $\tC_n$ except $s_i$.

The generators of $St_n$ are $s_1, \ldots, s_{n-1}$ acting as above, as well as
$$
g_0(x_1,x_2,\ldots,x_{n-1},x_n) := (x_n - 2,x_2,\ldots,x_{n-1},x_1 + 2)
$$
and
$$
g_n = (s_n \cdots s_0)^{n+4},\qquad
(s_n \cdots s_0)(x_1,x_2,\ldots,x_n) := (x_2,\ldots,x_n,x_1 + 2).
$$
Thus $g_0,s_1,\ldots,s_{n-1}$ are reflections, while $g_n$ is a cyclic permutation 
of coordinates combined with a translation.
The reflecting hyperplanes for $g_0,s_1,\ldots,s_{n-1}$ are
$$
H'_0 := \{(x_1,\ldots,x_n)\in\bbr^n\,|\,x_n = x_1 + 2\},
$$
$$
H_i := \{(x_1,\ldots,x_n)\in\bbr^n\,|\,x_i = x_{i+1}\}\qquad(1\le i\le n-1).
$$

\begin{observation}\label{t.Delta_n-1}
The subgroup $\langle g_0,s_1,\ldots,s_{n-1} \rangle$ of $\tC_n$ 
is an affine Weyl group of type $\tA_{n-1}$, and has
$$
H := \{(x_1,\ldots,x_n)\in\bbr^n\,|\,x_1 + \ldots + x_n =0\}
$$
as an invariant subspace.
A fundamental region for its action on $H$ is the $(n-1)$-dimensional simplex
$\Delta_{n-1}$ with vertices $w_0,\ldots,w_{n-1}\in H$, where
$$
w_i := \frac{2}{n}(\underbrace{i-n,\ldots,i-n}_{i},
                   \underbrace{i,\ldots,i}_{n-i})\qquad(0\le i\le n-1).
$$
\end{observation}

Now note that 
$$
(s_n \cdots s_0)(w_i) = e + w_{i-1}\qquad(0\le i\le n-1)
$$
with the index $i-1$ interpreted modulo $n$, and where 
$$
e = \frac{2}{n}(1,\ldots,1) \in H^{\perp}.
$$
Therefore $g_n$ acts as a linear transformation of $H$ 
preserving $\Delta_{n-1}$, combined with a translation by the vector 
$(n+4)e\in H^{\perp}$.

\begin{observation}\label{t.Fund_2}
A fundamental region for the action of $St_n$ on $\bbr^n$ is the prism
$$
Fund_2 = \Delta_{n-1} \times I := \{w+t(1,\ldots,1)\,|\,
                            w\in\Delta_{n-1},\,0\le t\le {2(n+4)/n}\}.
$$
\end{observation}

Return now to the action of $\tC_n$ on $CTFT(n)$. 
Each of the generators of $St_n$ clearly stabilizes the canonical colored
star triangulation $T_0$ defined immediately before Theorem~\ref{t.stabilizer}, 
so that $St_n$ is contained in the stabilizer of $T_0$ under the 
$\tC_n$-action on $CTFT(n)$. In order to show that $St_n$ is actually equal 
to this stabilizer, it suffices to show that both subgroups have the same 
finite index in $\tC_n$. 
The index of the stabilizer is the size of the orbit of
$T_0$, namely (by Proposition~\ref{t.action}) the number of colored
triangulations, $\#CTFT(n)$. The index of $St_n$ in $\tC_n$ is the
quotient of volumes $vol(Fund_2)/vol(Fund_1)$. By
Corollary~\ref{t.nTFT} 
it thus suffices to show that
$$
vol(Fund_2)/vol(Fund_1) = (n+4)\cdot 2^{n}.
$$
This indeed follows from the following computations, 
using the well-known formula for the volume of a $k$-dimensional simplex $\Delta$
with vertices $v_0,\ldots,v_k\in \bbr^k$:
$$
vol(\Delta) = \frac{1}{k!} \cdot [\det(\langle v_i-v_0,v_j-v_0 \rangle)_{1\le i,j \le k}]^{1/2}, 
$$
where $\langle \cdot, \cdot \rangle$ is the standard inner product on $\bbr^k$.

\begin{claim}\label{t.vol_Fund_1}
$$
vol(Fund_1) = \frac{1}{n!} \cdot \det(A)^{1/2},
$$
where, following Observation~\ref{t.Fund_1},
$$
A := (a_{ij})\in\bbr^{n \times n},
$$
$$
a_{ij} := \langle v_i- v_0, v_j - v_0 \rangle
        = min(i,j)\qquad(1\le i,j \le n).
$$
\end{claim}

\begin{claim}\label{t.vol_Fund_2}
$$
vol(Fund_2) = [2(n+4)/n]\cdot n^{1/2} \cdot \frac{1}{(n-1)!} \cdot \det(B)^{1/2},
$$
where, following Observations~\ref{t.Delta_n-1} and~\ref{t.Fund_2},
$$
B := (b_{ij})\in\bbr^{(n-1) \times (n-1)},
$$
$$
b_{ij} := \langle w_i - w_0, w_j - w_0 \rangle
        = \frac{4}{n} \cdot min(i,j) \cdot min(n-i,n-j)\qquad(1\le i,j \le n-1).
$$
\end{claim}

\begin{proof}
For $i\le j$,
\begin{eqnarray*}
b_{ij} &=& \langle w_i - w_0, w_j - w_0 \rangle \\
       &=& \frac{4}{n^2} \cdot 
       [i \cdot (i-n) \cdot (j-n) + (j-i) \cdot i \cdot (j-n) + (n-j) \cdot i \cdot j] \\
       &=& \frac{4}{n^2}\cdot [i \cdot (j-n)^2 + (n-j) \cdot i \cdot j] \\
       &=& \frac{4}{n^2}\cdot i \cdot (n-j) \cdot n.
\end{eqnarray*}
\end{proof}

\begin{claim}\label{t.det}
$$
\det(A) = 1
$$
and
$$
\det(B) = {\left(\frac{4}{n}\right)}^{n-1} \cdot n^{n-2} = 4^{n-1} \cdot n^{-1}.
$$
\end{claim}

\begin{proof}
By elementary row operations (subtracting row $i-1$ from row $i$, for $2\le i\le n$),
the $n\times n$ matrix $A = (min(i,j))$ can be transformed into an upper triangular matrix 
with $1$-s in and over the main diagonal, so that $\det(A) = 1$.

By similar operations, the $(n-1)\times(n-1)$ matrix $(n/4) \cdot B = (min(i,j) \cdot min(n-i,n-j))$ 
can be transformed into the matrix $C=(c_{ij})$ with 
$$
c_{i,j} = \begin{cases}
1 \cdot (n-j),& \mbox{\rm if } i\le j;\\
j \cdot (-1),&  \mbox{\rm if } i > j.
\end{cases}
$$
Subtracting row $n-1$ from all the other rows we get 
the matrix $D=(d_{ij})$ with 
$$
d_{i,j} = \begin{cases}
n,&  \mbox{\rm if } 1\le i\le j\le n-2;\\
0,&  \mbox{\rm if } 1\le j < i\le n-2;\\
0,&  \mbox{\rm if } 1\le i\le n-2 \mbox{\rm\ and } j=n-1;\\
-j,&  \mbox{\rm if } i = n-1 \mbox{\rm\ and } 1\le j\le n-2;\\
1,&  \mbox{\rm if } i = j = n-1.
\end{cases}
$$
It follows that $\det(C) = \det(D) = n^{n-2}$ and $\det(B) = 4^{n-1} \cdot n^{-1}$.
\end{proof}

\begin{claim}
$$
vol(Fund_2)/vol(Fund_1) = (n+4)\cdot 2^{n}= \# CTFT(n).
$$
\end{claim}

\begin{proof}
By Claims~\ref{t.vol_Fund_1}, \ref{t.vol_Fund_2} and~\ref{t.det},
$$
vol(Fund_1) = \frac{1}{n!}
$$
while
$$
vol(Fund_2) = 2(n+4)n^{-1/2} \cdot \frac{1}{(n-1)!} \cdot 2^{n-1}n^{-1/2} 
            = \frac{1}{n!} \cdot 2^n(n+4). 
$$
\end{proof}

This completes the proof of Theorem~\ref{t.stabilizer}.

\end{proof}


\subsection{Coset Representatives}

The stabilizer $St_n$ of the canonical colored star triangulation
$T_0$ is not a parabolic subgroup of $\tC_n$. However,
it will be shown that
a distinguished set of representatives of $St_n$ in $\tC_n$ forms an interval 
in the weak order on $\tC_n$.

\bigskip

For $0\le i\le n$ denote 
$a_i:=s_i s_{i-1} \cdots s_0 \in \tC_n$  and  
$b_i:=s_{n-i} s_{n-i+1} \cdots s_n \in \tC_n$.

\begin{proposition}\label{R_n}
Each of the sets
$$
R_n:=\{a_{0}^{\epsilon_0} a_{1}^{\epsilon_1}\cdots
a_{n-1}^{\epsilon_{n-1}} a_n^{\epsilon_n} :\ \epsilon_i\in\{0,1\}\
(0\le i<n) \ \text{\rm and }\ 0\le \epsilon_n< n+4 \}
$$
$$
R'_n:=\{b_{0}^{\epsilon_0} b_{1}^{\epsilon_{1}}\cdots
b_{n-1}^{\epsilon_{n-1}} b_n^{\epsilon_n} :\ \epsilon_i\in\{0,1\}\ 
(0\le i<n) \ \text{\rm and }\ 0\le \epsilon_n< n+4 \}
$$
forms a complete list of representatives of the left cosets of $St_n$ in $\tC_n$.
\end{proposition}

\begin{proof} 
Since $\# R_n\le (n+4)\cdot 2^{n}$, in order to prove
that $R_n$ forms a complete list of coset representatives it
suffices to prove that for every $T\in CTFT(n)$ there exists an
element $r\in R_n$ such that $r T_0=T$, where $T_0$ is the
canonical colored star triangulation. By
Observation~\ref{bijection-varphi}, it suffices to prove that for
every vector ${\rm v}=(v_0,\dots,v_n)\in \bbz_{n+4}\times
\bbz_2^n$ there exists $r\in R_n$ such that $\varphi(r T_0)={\rm
v}$. Indeed, by Observation~\ref{action-on-vectors},
$$
\varphi (a_0^{v_{n}} a_1^{v_{n-1}}\cdots a_{n-1}^{v_{1}}
a_n^{-v_0{\rm{mod}} (n+4)} T_0)=(v_0,\dots,v_n).
$$

The proof for $R'_n$ is similar.
\end{proof}

\medskip

Let $\ell(w)$ be the length of an element $w\in \tC_n$ with
respect to Coxeter generating set,  that is,
$$
\ell(w):=\min\{\ell:\ w = s_{i_1} s_{i_2} \cdots s_{i_\ell},\
 s_{i_j}\in\{s_0,\dots,s_n\}\ (\forall j) \}.
$$

\begin{claim}\label{length-Rn}
For every $r=a_0^{\epsilon_0}\cdots a_n^{\epsilon_n}\in R_n$
$$
\ell(r)=\sum\limits_{j=0}^n (j+1) \epsilon_j=\sum\limits_{j=0}^n
\sum\limits_{i=j}^n \epsilon_i.
$$
\end{claim}

\begin{proof}
Notice that for every $0\le i <n$, $a_{i+1}^{\epsilon_{i+1}}$ is a
representative of shortest length of a right coset of the
parabolic subgroup $\langle s_0,\dots,s_i\rangle$ in $\langle
s_0,\dots,s_i, s_{i+1}\rangle$. The Claim follows, by induction,
from the length-additivity property of parabolic subgroups in
Coxeter groups~\cite[\S 1.10]{Hum} \cite[\S 2.4]{BB}.
\end{proof}


\medskip

The following lemma plays a key role in understanding the
structure of $R_n$ (Proposition~\ref{lattice}) and of the colored
flip-graph (Propsosition~\ref{graph-description}).

\begin{lemma}\label{graph-description0}
For every $r=a_0^{\epsilon_0}\cdots a_n^{\epsilon_n}\in R_n$
 and a Coxeter
generator $s_i$ of $\widetilde{C}_n$ exactly one of the following
holds:
\begin{itemize}
\item[1.] $s_i r\in R_n$.

\item[2.] $s_i r\in r St_n$.

\item[3.] \begin{itemize}\item[(i)] $i=n$, $\epsilon_{n-1}=1$ and
$\epsilon_n=n+3$. Then $s_n r\in a_0^{\epsilon_0}\cdots
a_{n-2}^{\epsilon_{n-2}} St_n$. \item[(ii)]  $i=n$,
$\epsilon_{n-1}=0$ and $\epsilon_n=0$. Then $s_n r\in
a_0^{\epsilon_0}\cdots a_{n-2}^{\epsilon_{n-2}} a_{n-1}
a_n^{n+3}St_n$.

\end{itemize}
\end{itemize}
%
%
%
%
\end{lemma}

\begin{corollary}\label{left-descents}
For every $s_i\in S$ and $r=a_0^{\epsilon_0}\cdots
a_n^{\epsilon_n}\in R_n$
$$
\ell(s_i r)<\ell(r) \Longleftrightarrow \epsilon_{i-1}=0\
\text{\rm and}\ \epsilon_i>0,
$$
where $\epsilon_0:=0$.
\end{corollary}

For proofs of Lemma~\ref{graph-description0} and
Corollary~\ref{left-descents} see Appendix
(Section~\ref{appendix1}).

\bigskip

\noindent Denote
$$
w_o:= a_0a_1\cdots a_{n-1}a_n^{n+3}
$$
the longest element in $R_n$.

\begin{proposition}\label{lattice}
$R_n$ is a self-dual 
lower interval $\{w\in \widetilde{C}_n:\ id\le w\le w_o\}$
in the left weak order on $\widetilde{C}_n$;
hence it forms a graded lattice. 
\end{proposition}

\begin{proof}
By Corollary~\ref{left-descents}, for every $r\in R_n$ and $s_i\in
S$, $\ell(s_i r)<\ell(r)$ implies that
$r=
\cdots a_{i-1}^0 a_i^{\epsilon_i}\cdots $
($\epsilon_i>0$) for some $0\le i \le n$, thus
$$
s_i r=
\cdots a_{i-1} a_i^{\epsilon_i-1}\cdots \in
R_n.
$$
It follows that $R_n$ is an interval in the left weak order.

Self-duality follows from the identity 
$$
rw_0=a_{0}^{1-\epsilon_0} a_{1}^{1-\epsilon_1}\cdots
a_{n-1}^{1-\epsilon_{n-1}} a_n^{n+3-\epsilon_n} 
$$
for all $r=a_0^{\epsilon_0}\cdots a_n^{\epsilon_n}\in R_n$.

\end{proof}

\begin{remark}
Since $R_n$ is an interval in the left weak order the rank of an
element is given by its Coxeter length. Thus the rank generating
function is
$$
(1+q)(1+q^2)\cdots
(1+q^n)(1+q^{n+1}+q^{2(n+1)}+\cdots+q^{(n+3)(n+1)}),
$$
(not necessarily unimodal).
\end{remark}

\begin{lemma}\label{dominance}
For every pair of elements in $R_n$
$$
a_0^{\epsilon_0}\cdots a_n^{\epsilon_n}<a_0^{\delta_0}\cdots
a_n^{\delta_n}
$$
in the left weak order if and only if
$$
(\epsilon_n,\dots,\epsilon_0)<(\delta_n,\dots,\delta_0)
$$
in the dominance order; i.e., $\sum\limits_{i=k}^n
\epsilon_i<\sum\limits_{i=k}^n \delta_i$ for all $0\le k\le n$.
\end{lemma}

\begin{proof}
By Corollary~\ref{left-descents}, the lemma holds for the covering
relation. Proceed by induction on the length of the chain.
\end{proof}

For every pair of elements $r,s\in R_n$ denote by $r\wedge s$
their join and by $r\vee s$ their meet in the weak order on
$\widetilde{C}_n$. Lemma~\ref{dominance} implies

\begin{corollary}\label{join-meet}
For every pair of elements in $R_n$
$$
a_0^{\epsilon_0}\cdots a_n^{\epsilon_n}\ \wedge \
a_0^{\delta_0}\cdots a_n^{\delta_n}= a_0^{\alpha_0}\cdots
a_n^{\alpha_n},
$$
where
$$
\alpha_k:= \min\{\sum\limits_{i=k}^n
\epsilon_i,\sum\limits_{i=k}^n
\delta_i\}-\min\{\sum\limits_{i=k+1}^n
\epsilon_i,\sum\limits_{i=k+1}^n \delta_i\}  \qquad (0\le k\le n),
$$
and
$$
a_0^{\epsilon_0}\cdots a_n^{\epsilon_n}\ \vee \
a_0^{\delta_0}\cdots a_n^{\delta_n}= a_0^{\beta_0}\cdots
a_n^{\beta_n},
$$
where
$$
\beta_k:= \max\{\sum\limits_{i=k}^n \epsilon_i,\sum\limits_{i=k}^n
\delta_i\}-\max\{\sum\limits_{i=k+1}^n
\epsilon_i,\sum\limits_{i=k+1}^n \delta_i\} \qquad (0\le k\le n).
$$
\end{corollary}

It follows that

\begin{corollary}\label{modular}
$R_n$ forms a modular lattice with respect to the weak order;
namely, for every $r,s\in R_n$
$$
\ell(r\vee s)+\ell(r\wedge s)=\ell (r)+\ell(s).
$$
\end{corollary}

It should be noted that the weak order on $\widetilde{C}_n$ is not
modular.

\begin{proof}
Combining Corollary~\ref{join-meet} with Claim~\ref{length-Rn}
yields
$$
\ell(r\vee s)=\sum\limits_{j=0}^n \sum\limits_{i=j}^n \beta_i
= \sum\limits_{j=0}^n \max\{\sum\limits_{i=j}^n \epsilon_i ,
\sum\limits_{i=j}^n \delta_i \},
$$
and, similarly,
$$
\ell(r\wedge s)
= \sum\limits_{j=0}^n \min\{\sum\limits_{i=j}^n \epsilon_i ,
\sum\limits_{i=j}^n \delta_i \}.
$$
Hence
$$
\ell(r\wedge s)+\ell(r\vee s)= \sum\limits_{j=0}^n
\max\{\sum\limits_{i=j}^n \epsilon_i , \sum\limits_{i=j}^n
\delta_i \}+ \sum\limits_{j=0}^n \min\{\sum\limits_{i=j}^n
\epsilon_i , \sum\limits_{i=j}^n \delta_i \}
$$
$$
=\sum\limits_{j=0}^n \sum\limits_{i=j}^n (\epsilon_i+\delta_i)
=\ell(r)+\ell(s).
$$

\end{proof}

\section{The Flip Graph: Algebraic Description}


The colored flip graph $\Gamma_n$ is isomorphic to the Schreier
graph of the cosets of $St_n$ in $\widetilde{C}_n$ with respect to
the Coxeter generating set $\{s_0,\dots,s_n\}$. Furthermore,
fixing a set of coset representatives we can get an explicit
description of $\Gamma_n$.

\begin{proposition}\label{graph-description}
The colored flip graph $\Gamma_n$ is isomorphic to the
graph whose vertices are the elements in $R_n$; two distinct
elements $r_1,r_2\in R_n$ forms an edge if their quotient is a
Coxeter generator of $\widetilde{C}_n$ or 
they are of the form $(v,va_{n-1}a_n^{n+3})$, for any
$v=a_0^{\epsilon_0}\cdots a_{n-2}^{\epsilon_{n-2}}$.
\end{proposition}

In other words, the flip graph 
is obtained from the (undirected) Hasse diagram $\Sigma_n$ of the
left weak order on $R_n$ 
by adding the edges $(v,va_{n-1}a_n^{n+3})$, for any
$v=a_0^{\epsilon_0}\cdots a_{n-2}^{\epsilon_{n-2}}$.

\begin{proof}
Proposition~\ref{graph-description} is an immediate consequence of
Lemma~\ref{graph-description0}.
\end{proof}

\begin{observation}\label{aut0}
A right multiplication by $a_n$ is an automorphism of the colored
flip graph $\Gamma_n$.
\end{observation}

\begin{proof}
A rotation by $\frac{2\pi}{n+4}$ of the colored triangulation
$a_0^{\epsilon_0}\cdots a_{n-1}^{\epsilon_{n-1}}a_n^t T_0$ gives
the triangulation
 $a_0^{\epsilon_0}\cdots
a_{n-1}^{\epsilon_{n-1}}a_n^{t+1\, ({\rm mod}\, n+4)}T_0$.
\end{proof}

\medskip

For every pair $\pi,\sigma\in R_n$ let
$\dist_{\Gamma_n}(\pi,\sigma)$ be the distance between $\pi T_0$
and $\sigma T_0$ in $\Gamma_n$.

It follows from Observation~\ref{aut0} that
\begin{corollary}\label{rotation0}
For every pair $r, s\in R_n$ and an integer $t$
$$
\dist_{\Gamma_n}(r,s)=\dist_{\Gamma_n}(r a_n^t,s a_n^t).
$$
\end{corollary}

\section{The Flip Graph: Diameter}\label{Diameter}


Denote by $\diam(\Gamma_n)$ the diameter of the colored flip graph
$\Gamma_n$.

\begin{theorem}\label{diameter}
For every $n\ge 3$
$$
\diam(\Gamma_n) = \frac{(n+1)(n+4)}{2}.
$$
\end{theorem}

\subsection{Proof of Theorem~\ref{diameter}}\

The proof relies on the intimate relation between the colored flip
graph $\Gamma_n$ and the Hasse diagram of the weak order on $R_n$,
see Proposition~\ref{graph-description} and comment afterwards.
The upper bound (Lemma~\ref{upper-bound}) is obtained by combining
the properties of the weak order on $R_n$ with the invariance of
the flip graph under rotation. The grading of the Hasse diagram
together with Proposition~\ref{graph-description} implies a lower
bound (Lemma~\ref{lower-bound}).

\subsubsection{Distance}

For a graph $G$ denote by $\dist_G(v,u)$ the distance (i.e., the
length of the shortest path) between the vertices $u$ and $v$. We
begin with a general lemma.

\begin{lemma}\label{hasse-distance}
Let $P$ be a modular lattice. Let $\ell$ be its rank function and
$\Sigma$ its Hasse diagram. Then for every pair $r,s\in P$
$$
\dist_{\Sigma}(r,s)=\ell (r\vee s)-\ell(r \wedge s).
$$
\end{lemma}

\begin{proof}
If there is a shortest path between $r$ and $s$ with at most one
pick (local maximum), then by the modularity
$$
\dist_{\Sigma}(r,s)=2\ell(r \vee s)-\ell(r)-\ell(s)=\ell (r\vee
s)-\ell(r \wedge s).
$$
Given a shortest path from $r$ to $s$ with $k>1$ picks let $v,w$
be two consequent picks in the path.  There is a unique local
minimum $z$ in the path from $v$ to $w$. By the minimality of the
length of the path, $z=v \wedge w$. By the modularity we can
replace the segment from $v$ to $w$ through the meet $z$ by a path
through $v\wedge w$ and obtain a path of same length and $k-1$
picks. Proceed by recursion to get a shortest path with one pick.

\end{proof}


\begin{lemma}\label{gamma-distance}
For every pair $r=\prod\limits_{i=0}^n a_i^{\epsilon_i},
s=\prod\limits_{i=0}^n a_i^{\delta_i}\in V(\Gamma_n)=R_n$
$$
\dist_{\Gamma_n}(r,s)= $$ $$\min\{\ell (ra_n^{-\epsilon_n}\vee
sa_n^{-\epsilon_n})-\ell(ra_n^{-\epsilon_n} \wedge
sa_n^{-\epsilon_n}), \ell (ra_n^{-\delta_n}\vee
sa_n^{-\delta_n})-\ell(ra_n^{-\delta_n} \wedge
sa_n^{-\delta_n})\}.
$$
\end{lemma}

\begin{proof}
Let $C_n$ be a cycle of length $n+4$ whose set of vertices is
$\{u_i:\ 0\le i< n+4\}$ and edges $(u_i,u_{(i+1)\rm{ mod }(n+4)})$
for every $0\le i< n+4$. Consider the map
$\rho:\Gamma_n\longrightarrow C_n$, defined by
$\rho(a_0^{\epsilon_0}\cdots a_{n-1}^{\epsilon_{n-1}}a_n^i):=u_i$.
By Proposition~\ref{graph-description},  $\rho$ is a graph
homomorphism.
Let $U_i$ be the pre-image of $u_i$, i.e.
$$
U_i:=\rho^{-1}(u_i)=\{a_0^{\epsilon_0}\cdots
a_{n-1}^{\epsilon_{n-1}}a_n^i:\ \epsilon_j\in\{0,1\} \rm{\ for\
all\ } 0\le j<n\} \qquad (0\le i<n+4).
$$
Notice that the subgraph of $\Gamma_n$ induced by $U_i$ is
isomorphic to the undirected Hasse diagram of the weak order on
$U_i$.

We first claim that for
any $r,s\in R_n$ 
a shortest path from $r$ to $s$ in $\Gamma_n$ does not contain a
sequence of the form $v_1,\dots,v_k$, where
$v_1=\prod\limits_{j=0}^{n-1}
a_j^{\mu_j}a_n^i,v_k=\prod\limits_{j=0}^{n-1} a_j^{\nu_j}a_n^i\in
U_i$ for some $i$ and $v_2,\dots,v_{k-1}\in U_j$ for $j=(i\pm
1)\mod (n+4)$. If there is such a shortest path, then by
Corollary~\ref{rotation0}, we may assume that $i=0$ and $j=1$;
namely $v_1, v_k\in U_0$ and $v_1,v_2,\dots,v_{k-1}\in U_1$.
By assumption of the length minimality of the path
$$
\dist_{U_0} (v_1,v_k)\ge 2 +\dist_{U_1}(v_2,v_{k-1}).
$$
On the other hand, by Proposition~\ref{graph-description},
$\mu_{n-1}=\nu_{n-1}=1$, $v_2=\prod\limits_{j=0}^{n-2}
a_j^{\mu_j}a_n$, and $v_{k-1}=\prod\limits_{j=0}^{n-2}
a_j^{\nu_j}a_n$. Hence, by Lemma~\ref{hasse-distance} together
with Corollary~\ref{join-meet},
$$
\dist_{U_0} (v_1,v_k)=\dist_{U_1}(v_2,v_{k-1}).
$$
Contradiction.

We deduce that the $\rho$-image of the shortest path between any
pair $r,s\in U_i$ for some $i$ is either of length zero or a
multiple of a full cycle. But it cannot be a multiple of a full
cycle since a pre-image of a full cycle is of length at least
$$
\ell(a_n^{n+3})-\ell(a_0\cdots a_{n-1})= \frac{(n+1)(n+5)}{2}.
$$
On the other hand, since $U_i$ is a modular lattice, the diameter
of $U_i$ is the difference between the lengths of the top and
bottom elements in $U_i$. That is
\begin{equation}\label{U_i-diameter}
\diam(U_i)=\ell(a_0\cdots a_{n-1}a_i)-\ell(a_i)={n+1\choose 2}.
\end{equation}
One concludes that  for any pair $r,s\in U_i$ the shortest path is
contained in the modular lattice $U_i$, so the lemma holds for
such a pair.

If $r\in U_i$, $s\in U_j$ and $i< j$ then by the above arguments
the $\rho$-image of the shortest path between $r$ and $s$ is one
of the two intervals from $u_i$ to $u_j$ in the cycle. By
Corollary~\ref{rotation0}, a right multiplication (by
$a_n^{-\epsilon_n}$ if the image contains $u_{i+1}$ or by
$a^{-\delta_n}$ otherwise) maps the shortest path to a shortest
path in the modular lattice $R_n$. Lemma~\ref{hasse-distance}
completes the proof.

\end{proof}

\subsubsection{Diameter: Upper Bound}

In this subsection we prove

\begin{lemma}\label{upper-bound}
For every $n\ge 3$
$$
\diam(\Gamma_n)\le \frac{(n+1)(n+4)}{2}.
$$
\end{lemma}

\begin{proof}


By the lattice property and modularity of $R_n$
(Corollary~\ref{modular}) together with Claim~\ref{length-Rn} and
Corollary~\ref{join-meet}, for every $r=a_0^{\epsilon_0}\cdots
a_n^{\epsilon_n}, s=a_0^{\delta_0}\cdots a_n^{\delta_n}\in R_n$
\begin{equation}\label{up}
\dist_{\Gamma_n}(r,s)
= \ell(r\vee s)-\ell(r\wedge s)=\sum\limits_{j=0}^n
|\sum\limits_{i=j}^n (\epsilon_i-\delta_i)|.
\end{equation}

\noindent If $\epsilon_n=\delta_n$ then
there exists $0\le i< n+4$ such that $r,s\in U_i$. Then by
(\ref{U_i-diameter}),
$$
\dist_{\Gamma_n}(r,s)\le {n+1\choose 2}.
$$
If $\epsilon_n\ne \delta_n$ then by Corollary~\ref{rotation0}, 
we may assume, without loss of generality, that $\delta_n=0$.
Also, by note that Corollary~\ref{rotation0},
$\dist_{\Gamma_n}(r,s)= \dist (ra_n^{-\epsilon_n},
sa_n^{-\epsilon_n}).$ Now, by Lemma~\ref{gamma-distance} together
with (\ref{up}) and the assumption $\delta_n=0$,

\begin{equation}\label{gamma-distance-eq}
\dist_{\Gamma_n}(r,s) = \min\{ \sum\limits_{j=0}^n
|\epsilon_n+\sum\limits_{i=j}^{n-1} (\epsilon_i-\delta_i)|,
\sum\limits_{j=0}^n |n+4-\epsilon_n-\sum\limits_{i=j}^{n-1}
(\epsilon_i-\delta_i)|\}.
\end{equation}
For $0\le j\le n$ denote $x_j:=\epsilon_n+\sum\limits_{i=j}^n
(\epsilon_i-\delta_i) - \frac{n+4}{2}$. Then
$$
\dist_{\Gamma_n}(r,s) =
\min\{ \sum\limits_{j=0}^n |\frac{n+4}{2}+x_j|, 
       \sum\limits_{j=0}^n |\frac{n+4}{2}-x_j|\},
$$
where, by definition, $(i)$ $-\frac{n+4}{2}\le x_n< \frac{n+4}{2}$
and $(ii)$
$|x_{j+1}-x_j|\le 1$. 

By $(i)$, $\frac{n+4}{2}+x_n\ge 0$. Combining this with $(ii)$
implies that if $\frac{n+4}{2}+x_j$ is negative for some $j$, then
there exists $0\le j_o\le n$, such that $\frac{n+4}{2}+x_{j_o}= 0$.
Then, by $(ii)$, for every $0\le j\le n$, $|\frac{n+4}{2}+x_j|\le
|j-j_o|$. Hence $\sum\limits_{j=0}^n
|\frac{n+4}{2}+x_j|\le 
{n+1\choose 2}$. So, we may assume that $\frac{n+4}{2}+x_j$ is
positive for all $0\le j\le n$. By a similar reasoning (regarding
the second sum), we may assume that $\frac{n+4}{2}-x_j$ is 
positive for all $0\le j\le n$. Thus
$$
\dist_{\Gamma_n}(r,s) =\min\{ \sum\limits_{j=0}^n \frac{n+4}{2}+x_j,
\sum\limits_{j=0}^n \frac{n+4}{2}-x_j\} \le 
 \frac{(n+1)(n+4)}{2}.
$$
\end{proof}




\subsubsection{Diameter: Lower Bound}

In this subsection we prove

\begin{lemma}\label{lower-bound}
Let $n\ge 3$. For every $r\in R_n$ there exists an element $s\in
R_n$ such that
$$
\dist_{\Gamma_n}(r,s)\ge \frac{(n+1)(n+4)}{2}.
$$
In particular,
$$
\diam(\Gamma_n)\ge \frac{(n+1)(n+4)}{2}.
$$
\end{lemma}

\begin{proof}
Since the Hasse diagram $\Sigma_n$ on $R_n$ is graded by the
length function $\ell$, and since $\Gamma_n$ is obtained from the
$\Sigma_n$ by adding the edges $(v,va_{n-1}a_n^{n+3})$, for any
$v=a_0^{\epsilon_0}\cdots a_{n-2}^{\epsilon_{n-2}}$
(Proposition~\ref{graph-description}) it follows that for every
$r,s\in R_n$
\begin{equation}\label{lobound}
\dist_{\Gamma_n}(r,s)\ge \min\{ |\ell(s)-\ell(r)|,
\ell(a_{n-1}a_n^{n+3})+1-|\ell(s)-\ell(r)|\}.
\end{equation}
It follows that
$$
\diam(\Gamma_n)=\max\{\dist_{\Gamma_n}(r,s):\ r,s\in R_n\}
$$
$$
\ge \max\{ \min\{d, (n+1)(n+4)-d\}:\ 0\le d\le 3{n+2\choose 2}\},
$$
where $d:=|\ell(s)-\ell(r)|$, hence $0\le d\le \ell(w_o)={n\choose
2}+(n+3)(n+1)=3{n+2\choose 2}$. By Proposition~\ref{lattice}, for
any given $r\in R_n$, there exists an $s\in R_n$ of length
distance $\frac{(n+1)(n+4)}{2}$, completing the proof.

\end{proof}

\noindent Combining Lemma~\ref{upper-bound} with
Lemma~\ref{lower-bound} completes the proof of
Theorem~\ref{diameter}.

\qed

\subsection{Antipodes}\




Let $\phi:
CTFT(n)\longrightarrow CTFT(n)$ denote the map which reverse the
coloring of a triangle free triangulation; namely each color $i$
is replaced by $n-i$. Clearly, $\phi$ is an automorphism of the
colored flip graph $\Gamma_n$.
Furthermore,

\begin{proposition}\label{aut1}
For every $T\in CTFT(n)$ the flip distance between $T$ and
$\phi(T)$ is equal to $\diam(\Gamma_n)$.
\end{proposition}

To prove that we need the following Lemma.
Let $f$ be the natural bijection from $CTFT(n)$ to $R_n$: 
$f(T):=r$ if $rT_0=T$. Then


\begin{lemma}\label{varphi}
For every $r=a_0^{\epsilon_0}\cdots a_{n-1}^{\epsilon_{n-1}}
a_n^{\epsilon_n}$ with $\epsilon_i\in\{0,1\}$ $(0\le i< n)$ and
$0\le \epsilon_n < n+4$
$$
f^{-1} \phi
(rT_0)=\prod\limits_{i=0}^{n-1}a_i^{1-\epsilon_{n-1-i}} \cdot
a_n^m,
$$
where $m:=(
2+\sum\limits_{i=0}^{n}\epsilon_i) ({\rm mod}\ n+4)$.
\end{lemma}

\medskip


\noindent{\bf Proof of Proposition~\ref{aut1}.}
By Corollary~\ref{rotation0}, we may assume that $\epsilon_n=0$.
Then by Lemma~\ref{varphi},
$ f^{-1} \phi (rT_0)=a_0^{1-\epsilon_{n-1}}\cdots
a_{n-1}^{1-\epsilon_{0}} a_n^m$,
where $m=
2+\sum\limits_{i=0}^{n-1}\epsilon_i$.
By 
Claim~\ref{length-Rn},
$$
\ell(f^{-1} \phi
(rT_0))-\ell(r)=(2+\sum\limits_{i=0}^{n-1}\epsilon_i)(n+1)+\sum\limits_{i=0}^{n-1}
(i+1)(1-\epsilon_{n-1-i})- \sum\limits_{i=0}^{n-1} (i+1)\epsilon_i
$$ $$ = 2(n+1) + \sum\limits_{i=0}^{n-1}(i+1)+
\sum\limits_{i=0}^{n-1}\epsilon_i ((n+1)-(i+1)-(n-i))= \frac{(n+1)(n+4)}{2}.
$$
Hence by 
(\ref{lobound}), $\dist_{\Gamma_n}(f^{-1} \phi (rT_0),r)\ge \frac{(n+1)(n+4)}{2}$, 
so it is equal to 
 the diameter. \qed

\bigskip

Another antipode may be obtained by rotation. For an even $n$ let
$\psi$ denote the rotation of a triangle free triangulation $T\in
CTFT(n)$ by $\pi$ with respect to the center of $P_{n+4}$. Then

\begin{proposition}\label{aut2}
For every even $n$ and $T\in CTFT(n)$ the flip distance between $T$
and $\psi(T)$ is equal to $\diam(\Gamma_n)$.
\end{proposition}

\begin{proof}
Without loss of generality, $r=f^{-1}(T)=a_0^{\epsilon_0}\cdots
a_{n-1}^{\epsilon_{n-1}}$.  By the proof of
Observation~\ref{aut0}, for every $r\in R_n$ $ f^{-1} \psi (rT_0)=
r a_n^{(n+4)/ 2}.$ Hence, by Claim~\ref{length-Rn},
$$
\ell(f^{-1} \psi (rT_0))-\ell(r) = \frac{n+4}{2}(n+1).
$$
Combining this with (\ref{lobound}) yields
$$
\dist_{\Gamma_n}(f^{-1} \psi (rT_0),r)\ge \frac{n+4}{2}(n+1)= \diam(\Gamma_n).
$$

\end{proof}

\section{Final Remarks}


The colored flip-graph is bipartite; the bipartition is fixed by
the parity of the corresponding elements in $R_n$.

Recall the natural bijection $f: CTFT(n) \longrightarrow  R_n$, defined by 
$f(T):=r$ if $rT_0=T$.

\begin{defn} For every $T\in CTFT(n)$ associate a sign
$$
{\rm sign}(T):=(-1)^{\ell(f(T))}.
$$
\noindent A triangulation $T\in CTFT(n)$ is {\rm even}
if ${\rm sign}(T)=1$ and {\rm odd} otherwise.
\end{defn}

\begin{proposition}
The graph $\Gamma_n$ is bipartite; the flip operation changes the
sign.
\end{proposition}

\begin{proof}
By Proposition~\ref{graph-description}, the vertices of $\Gamma_n$
may be identified with elements in $R_n$, where
for every pair $r,s\in R_n$ $(r,s)$ is an edge in the colored flip
graph if and only if $rs^{-1}$ is a Coxeter generator  or equals
to $(a_{n-1}a_n^{n+3})^{\pm 1}$. Notice that for every $n$ the
length $\ell(a_{n-1}a_n^{n+3})=n+(n+1)(n+3)$ is odd. We conclude
that if $(r,s)$ is an edge then $r$ and $s$ differ by the parity
of their Coxeter length.
\end{proof}

\begin{proposition}
The number of even triangulations is equal to the number of odd
triangulations.
\end{proposition}

\begin{proof}
By definition of the sign, it suffices to show that there exists
an invertible map from $R_n$ to itself, which changes the parity
of the length. A left multiplication by $s_0$ 
is such a map.
\end{proof}


\bigskip

Hereby we mention (without proofs) some 
properties of the stabilizer $St_n$.

\begin{proposition}
The stabilizer $St_n$ is isomorphic to the direct product
$\widetilde{A}_n\otimes \bbz$.
\end{proposition}

\medskip

Even though $St_n$ is not a parabolic subgroup of
$\widetilde{C}_n$ the following remarkable property holds.

\begin{proposition}
Every coset of $St_n$ in $\widetilde{C}_n$ has a unique shortest
representative.
\end{proposition}

The set of shortest representatives 
may be
constructed from $R_n$ by a slight modification. Let $B(n,r):=
\{\pi\in \widetilde{C}_n:\ \ell(\pi)\le r\}$ be the ball of radius
$r$ in $\widetilde{C}_n$.

\begin{proposition}\label{Rn-list}
The set
$$
\hat R_n := \left(R_n\cap B(n,\frac{(n+1)(n+4)}{2})\right) \bigcup
\left(R_n\setminus B(n,\frac{(n+1)(n+4)}{2})\right) g_n^{-1}
$$
forms a complete list of shortest representatives of the left
cosets of $St_n$ in $\widetilde{C}_n$.
\end{proposition}

\medskip






Finally, the computation of the diameter of the flip graph of
uncolored  triangle-free triangulations involves a surprisingly
subtle optimization problem and will be addressed elsewhere.

\section{Appendix: Proofs of Lemma~\ref{graph-description0} and
Corollary~\ref{left-descents}}\label{appendix1}


\noindent{\bf Proof of Lemma~\ref{graph-description0}.}

First, notice that, by the braid relations of $\tC_n$, (letting
$a_{-1}:=id$)
\begin{equation}\label{sa2}
s_i a_j=a_j s_i\qquad (-2\le j-1<i\le n);
\end{equation}
\begin{equation}\label{sa1}
s_i a_i = a_{i-1} \ \ \ {\rm{and}}\ \ \ s_i a_{i-1} = a_i \qquad
(0\le i\le n)
\end{equation}
\begin{equation}\label{sa3}
s_i a_j=a_j s_{i+1} \qquad (0< i<j  \ {\rm{and}}\ 
i\ne n-1 );
\end{equation}
and
\begin{equation}\label{sa4}
 s_{n-1} a_n= a_n g_0,
\end{equation}
and
\begin{equation}\label{sa5}
 a_n s_1= g_0 a_n,
\end{equation}
where $g_0:= s_0 s_1 \cdots s_{n-2} s_n s_{n-1} s_n s_{n-2} \cdots
s_1 s_0$. Recall that $g_0\in St_n$ (see
Theorem~\ref{t.stabilizer}).

\smallskip

We proceed by cases analysis.

\noindent{\bf Case (a).} $ \epsilon_{i-1}=0$ and $\epsilon_i>0$
(where $\epsilon_{-1}:=0$).

\noindent 
By (\ref{sa1}) and (\ref{sa2}),
$$s_i r= s_i a_0^{\epsilon_0}\cdots a_{i-2}^{\epsilon_{i-2}}
a_{i}^{\epsilon_i} a_{i+1}^{\epsilon_{i+1}}\cdots
a_n^{\epsilon_n}= a_0^{\epsilon_0}\cdots a_{i-2}^{\epsilon_{i-2}}
s_i a_i a_i^{\epsilon_i-1} a_{i+1}^{\epsilon_{i+1}}\cdots
a_n^{\epsilon_n}
$$
$$
= a_0^{\epsilon_0}\cdots a_{i-2}^{\epsilon_{i-2}} a_{i-1}
a_i^{\epsilon_i-1} a_{i+1}^{\epsilon_{i+1}}\cdots
a_n^{\epsilon_n}
\in R_n.$$

\noindent{\bf Case (b).} $0<i<n$, $ \epsilon_{i-1}=1$ and
$\epsilon_i=0$, or $i=0$ and $\epsilon_0=0$. 

\noindent Since $s_i^2=1$, it follows from the analysis of the
previous case that in this case  $$s_i r=a_0^{\epsilon_0}\cdots
a_{i-2}^{\epsilon_{i-2}} a_{i} a_{i+1}^{\epsilon_{i+1}}\cdots
a_n^{\epsilon_n}\in R_n.$$

\noindent{\bf Case (c).} $i=n$.

\noindent If $\epsilon_{n-1}=1$ then $s_n r=
a_0^{\epsilon_0}\cdots
a_{n-2}^{\epsilon_{n-2}}a_n^{\epsilon_n+1}$. For $\epsilon_n<n+3$
this is  an element in $R_n$. If $\epsilon_n=n+3$ then, since
$a_n^{n+4}\in St_n$, $s_n r\in a_0^{\epsilon_0}\cdots
a_{n-2}^{\epsilon_{n-2}} St_n$.

\noindent If $\epsilon_{n-1}=0$ then $s_n r=
a_0^{\epsilon_0}\cdots a_{n-2}^{\epsilon_{n-2}} a_{n-1}
a_n^{\epsilon_n-1}$. This element is in $R_n$ if $\epsilon_n>0$,
and belongs to $a_0^{\epsilon_0}\cdots a_{n-2}^{\epsilon_{n-2}}
a_{n-1} a_n^{\epsilon_n+3}St_n$ otherwise.

\smallskip

\noindent{\bf Case (d).} $0<i<n$, $\epsilon_{i-1}=1$ and
$\epsilon_i=1$.

\noindent By the braid relations of $\tC_n$, for every $0<i<n$,
$s_i a_{i-1} a_i= a_i^2= a_{i-1} a_i s_1$. Hence
$$
s_i r=s_i a_0^{\epsilon_0}\cdots a_{i-2}^{\epsilon_{i-2}}
a_{i-1}a_i a_{i+1}^{\epsilon_{i+1}}\cdots
a_n^{\epsilon_n}=a_0^{\epsilon_0}\cdots a_{i-2}^{\epsilon_{i-2}}
s_i a_{i-1}a_i a_{i+1}^{\epsilon_{i+1}}\cdots a_n^{\epsilon_n}
$$
$$
=a_0^{\epsilon_0}\cdots a_{i-2}^{\epsilon_{i-2}} a_{i-1} a_i s_1
a_{i+1}^{\epsilon_{i+1}}\cdots a_n^{\epsilon_n} =
  r s_1\in r St_n
$$

\smallskip

\noindent{\bf Case (e).} $0<i<n$, $\epsilon_{i-1}=0$ and
$\epsilon_i=0$.

\noindent By (\ref{sa2})
$$s_i r= a_0^{\epsilon_0}\cdots
a_{i-2}^{\epsilon_{i+2}} s_i a_{i+1}^{\epsilon_{i+1}}\cdots
a_n^{\epsilon_n}= a_0^{\epsilon_0}\cdots a_{i-2}^{\epsilon_{i+2}}
a_{i+1}^{\epsilon_{i+1}}\cdots a_{n-1}^{\epsilon_{n-1}}
s_{i+k}a_n^{\epsilon_n},
$$
where $k:=\#\{j:\ i<j<n, \epsilon_j=1\}$.

If $i+k<n-\epsilon_n$ then, by (\ref{sa3}),
$s_{i+k}a_n^{\epsilon_n}= a_n^{\epsilon_n} s_{i+k+\epsilon_n}$, so
that $s_ir=r s_{i+k+\epsilon_n}$. Since $0<i+k+\epsilon_n<n$,
$s_{i+k+\epsilon_n}\in St_n$, thus $s_i r\in r St_n$.

If $i+k\ge n-\epsilon_n$, then by definition of $k$, $i+k=n-1$. By
(\ref{sa4}), (\ref{sa5}) and (\ref{sa3}),
$$
s_{i+k}a_n^{\epsilon_n}=a_n g_0 a_n^{\epsilon_n-1}= 
\begin{cases}
a_n^{\epsilon_n} g_0,&\ {\rm{if}}\ \epsilon_n=1,\\
a_n^{\epsilon_n} s_{\epsilon_n-1},&\ {\rm{if}}\ 1<\epsilon_n\le n,\\
a_n^{\epsilon_n} g_0,&\ {\rm{if}}\ \epsilon_n= n+1,\\
a_n^{\epsilon_n} s_{m-1},&\ {\rm{if}}\  \epsilon_n= n+m,\ m=2,3.\\
\end{cases}
$$
Hence $s_i r\in r St_n$.


\qed

\medskip

\noindent{\bf Proof of Corollary~\ref{left-descents}.}
The proof follows from the case-by-case analysis in the proof of
Lemma~\ref{graph-description0}. If $ \epsilon_{i-1}=0$ and
$\epsilon_i>0$ then $s_i r= s_i (\cdots a_{i-1}^0
a_i^{\epsilon_i}\cdots )= \cdots a_{i-1} a_i^{\epsilon_i-1}\cdots
$. By Claim~\ref{length-Rn}, $\ell(s_i r)<\ell(r)$. If
$\epsilon_{i-1}=1$ and $\epsilon_i=0$, by same argument $\ell(s_i
r>\ell(r)$. Similarly, for $i=n$ and $\epsilon_{n-1}=1$.
Otherwise, by Lemma~\ref{graph-description0}, $s_i r=r g$ for some
$g\in St_n$. By the length-additivity property~\cite[\S
1.10]{Hum}, $\ell (s_i r)=\ell(r)+\ell(g)\ge \ell(r)$.


\qed

\end{document}